\documentclass[12pt,reqno]{amsart}
\usepackage{amsfonts}
\usepackage{amssymb}
\usepackage{amsmath}
\usepackage{graphicx}

\newcommand{\cii}[1]{_{ {}_{ #1}}}
\newcommand{\av}[2]{\langle #1\rangle\cii {#2}}

\newcommand{\cF}{\mathcal{F}}
\newcommand{\R}{\mathbb{R}}
\newcommand{\A}[1]{A(#1)}
\newcommand{\tcF}{\A{\cF}}
\newcommand{\cl}{\mathrm{cl}}
\newcommand{\dfi}{\partial_{\mathrm{fixed}}}
\newcommand{\dfree}{\partial_{\mathrm{free}}}
\newcommand{\BMO}{\mathrm{BMO}}
\newcommand{\bmody}{\mathrm{BMO}^d([0,1]^n)} % Это обозначение диадического БМО на кубе размерности n

\newcommand{\eps}{\varepsilon}

\newcommand{\Class}{\boldsymbol{A}}
\newcommand{\vf}{\varphi}
\newtheorem{Th}{Theorem}
\newtheorem{Le}{Lemma}
\newtheorem{Def}{Definition}
\newtheorem{Prop}{Proposition}
\newtheorem{Cor}{Corollary}

\begin{document}
\title[Monotonic rearrangements]{Monotonic rearrangements of functions with small mean oscillation}
\author{Dmitriy M. Stolyarov}
\address{
Institute for Mathematics, Polish Academy of Sciences, Warsaw\\
P. L. Chebyshev Research Laboratory, St. Petersburg State University\\
St. Petersburg Department of Steklov Mathematical Institute, Russian Academy of Sciences (PDMI RAS).
}
\email{dstolyarov@impan.pl, dms@pdmi.ras.ru}
\urladdr{http://www.chebyshev.spb.ru/DmitriyStolyarov}

\author{Vasily I. Vasyunin}
\address{St. Petersburg Department of Steklov Mathematical Institute, Russian Academy of Sciences (PDMI RAS).}
\email{vasyunin@pdmi.ras.ru} 
\author{Pavel B. Zatitskiy}
\address{P. L. Chebyshev Research Laboratory, St. Petersburg State University\\
St. Petersburg Department of Steklov Mathematical Institute, Russian Academy of Sciences (PDMI RAS).
}
\email{paxa239@yandex.ru}
\urladdr{http://www.chebyshev.spb.ru/pavelzatitskiy}
%\webpage{http://www.chebyshev.spb.ru/pavelzatitskiy}

\thanks{The research is supported by the grant of the Russian Science Foundation (project 14-21-00035).}

\begin{abstract}
We obtain sharp bounds for the monotonic rearrangement operator from ``dyadic-type'' classes to ``continuous''. In particular, for the~$\BMO$ space and Muckenhoupt classes.  

The idea is to connect the problem with a simple geometric construction named $\alpha$-extension.
\end{abstract}
\maketitle 

\section{Introduction}
The~$\BMO$ space has many nice properties. One of them is that the monotonic rearrangement operator is bounded on this space (see~\cite{GR,BSV}). In other words, the inequality
\begin{equation*}
\|f^*\|_{\BMO} \leq c\|f\|_{\BMO}
\end{equation*} 
holds true with some constant~$c$. It is not hard to see that~$c \geq 1$. Soon it was noticed that~$c = 1$ when the dimension of the underlying space is one (see~\cite{Klemes}).

The same boundedness is also present when the~$\BMO$ space is replaced by its relatives: the Muckenhoupt classes or the Gehring classes (see~\cite{Wik1,Wik2} for the Muckenhoupt class and~\cite{FM} for the Gehring class). And again, if the underlying space is an interval, then the constant in the corresponding inequality equals to one, i.e.
\begin{equation}\label{A_p_OneDimensional}
[f^*]_{A_p} \leq [f]_{A_p} 
\end{equation}
(see~\cite{BSW} for~$A_1$ and~\cite{Korenovskii} for the general case). On the other hand, see~\cite{BSW} for an example showing that inequality~\eqref{A_p_OneDimensional} does not hold for higher-dimensional case.

In~\cite{SZ} the authors developed a setting that unifies the three cases (and, moreover, covers a more general situation described in \cite{5A} for the case of a related extremal problem) and gave a proof of an inequality that generalizes~\eqref{A_p_OneDimensional} to the setting of that paper. The proof relied on passing to a certain class of martingales. 

It seems a rather difficult problem to calculate the norm of the monotonic rearrangement operator in higher dimensions. Not able to solve it, we cope with a problem that lies towards it. Namely, we will calculate the aforementioned norm for the case when the space~$\BMO$ (or any other class of similar nature) is dyadic. The dyadic classes seem to be a step towards the higher-dimensional case not only in our problem, but, for example, in the problem of finding the sharp constant in the John--Nirenberg inequality (see~\cite{SV}). For numerous applications of monotonic rearrangements in variuos estimations see~\cite{DK,KS,MN1,MN2,RS} and references therein.

We briefly formulate the corollaries of our abstract considerations that concern the classical cases of the~$\BMO$ space and the Muckenhoupt class. Let $n\in \mathbb{N}$ and let $\mathcal{D}$ be the set of all dyadic subcubes of~$[0,1]^n$. Let us consider now the dyadic $\BMO$ space on $[0,1]^n$ with the quadratic seminorm: 
\begin{multline}\label{DyadicSpace}
\bmody=\{\vf\in L^1([0,1]^n)\colon \\
\|\vf\|^2_{\bmody}=\sup_{I\in\mathcal{D}}(\av{\vf^2}{I}-\av{\vf}{I}^2)<+\infty\}.
\end{multline}
If the supremum of the same value is taken over all subcubes of~$[0,1]^n$, then we obtain the usual (``continuous'')~$\mathrm{BMO}$ quadratic seminorm (in this paper, we consider only quadratic seminorms on~$\mathrm{BMO}$).

The monotonic rearrangement of a function~$\varphi$ from this space is a monotone (say, non-increasing) function~$\varphi^*$ on~$[0,1]$ with the same distribution as the function itself. The (non-linear) operator~$\varphi \mapsto \varphi^*$ is called the monotonic rearrangement operator.
\begin{Cor}\label{corBMO}
The monotonic rearrangement operator acts from the space~$\bmody$ to $\mathrm{BMO}([0,1])$ and its norm equals~$\frac{1+2^n}{2^{1+n/2}}$.
\end{Cor} 

Let us consider the dyadic Muckenhoult class~$A_2^d$ on $[0,1]^n$. A positive function $\vf$ on~$[0,1]^n$ is in $A_2^d$ with  the constant $Q$ if $\av{\vf}{I} \av{\vf^{-1}}{I} \leq Q$ for any $I \in \mathcal{D}$. Define the set $A_{2,Q}^d([0,1]^n)$ by the formula
\begin{equation*}
A_{2,Q}^d([0,1]^n)= \{\vf \in L^1([0,1]^n)\colon \sup_{I \in \mathcal{D}}\big(\av{\vf}{I} \av{\vf^{-1}}{I}\big) \leq Q\}.
\end{equation*}
 Again, $A_{2,Q}([0,1])$ stands for the set of functions on $[0,1]$, for which a similar supremum taken over all subintervals of~$[0,1]$ does not exceed~$Q$.

\begin{Cor}\label{corA2}
The monotonic rearrangement operator acts from $A_{2,Q}^d([0,1]^n)$ to $A_{2,Q'}([0,1])$ if and only if $Q' \geq \frac{Q(2^n+1)^2-(2^n-1)^2}{2^{n+2}}$.
\end{Cor}
A similar statement can be obtained for the~$A_p$ class,~$p \ne 2$, but there is no beautiful answer (it includes many solutions of implicit algebraic equations), so we do not dwell on this. An interested reader may calculate the sharp constant herself using Propositions~\ref{propap1} and~\ref{propap2} below.

Another motivation to write (and, we hope, to read) this paper is to demonstrate the strength of the martingale technique developed in~\cite{SZ}. More or less, the proof consists of an accurate manipulation with the definitions, a very simple geometric lemma from~\cite{SV}, and a martingale embedding theorem from~\cite{SZ}. 

In the next section, we state the main theorem in an abstract form (using the terminology from~\cite{5A,SV}) and prove it. The last section consists of the examples of specific classes (in particular, contains the proof of Corollaries~\ref{corBMO},~\ref{corA2}).

\section{Preliminaries and main theorem}
Fix unbounded open strictly convex domains $\Omega_0, \Omega_1 \subset \mathbb{R}^2$ satisfying the following conditions:
\begin{itemize}\label{cond}
\item $\cl\Omega_1\subset \Omega_0$;
\item (cone property) any ray lying inside $\Omega_0$ can be shifted to lay inside $\Omega_1$.
\end{itemize}
Put $\Omega=\cl(\Omega_0\setminus\Omega_1)$. In what follows we will consider only domains of this type and call them lenses. The set $\partial\Omega_0$ is called the fixed boundary of the lens $\Omega$ and is denoted by $\dfi\Omega$. The residual part of the boundary $\partial\Omega_1 = \partial\Omega \setminus \dfi \Omega$ is called the free boundary of $\Omega$ and is denoted by $\dfree\Omega$.

Recall a definition from~\cite{5A}.
\begin{Def}%\noindent{\bf Definition.}
Let $J\subset\R$ be an interval and $\vf \colon J\to\partial\Omega_0$ be a summable function.
We say that the function $\vf$ belongs to the class $\Class_\Omega$ if
$\av\vf{I}\in\Omega$ for every subinterval~$I\subset J$.
\end{Def}

Since the domain $\Omega_0$ is unbounded and strictly convex, there exists at least one straight line $\ell \subset \mathbb{R}^2$ such that the orthogonal projection $P_\ell$ onto this line is injective on $\partial{\Omega_0}$. A function $\vf \in \Class_\Omega$ is called monotone if the composition $P_\ell \circ \vf$ is monotone. The function $\vf^* \colon J \to \partial{\Omega_0}$ is the monotonic rearrangement of $\vf$ if it is monotone (say, non-increasing) and has the same distribution as~$\vf$.

Let $(X,\mathfrak A,\mu)$ be a standard probability space. For any integrable vector-valued function $\vf$ on  $(X,\mathfrak A,\mu)$ and for any subset $\omega \in \mathfrak A$ of positive measure we denote by $\av{\vf}{\omega}$ the average of $\vf$ over $\omega$, that is $\av{\vf}{\omega}=\frac{1}{\mu(\omega)}\int_\omega \vf d\mu$. 
 
We consider increasing discrete time filtrations $\cF=\{\cF_n\}_{n \geq 0}$ with finite algebras $\cF_n$ starting with trivial algebra $\cF_0=\{\varnothing, X\}$ such that~$\mathfrak{A}$ is generated by $\{\cF_n\}_{n \geq 0}$ mod 0. By L\'evy's zero-one law if $f \in L^1(X,\mathfrak A,\mu)$ and $\{F_n\}_{n \geq 0}$ is the martingale generated by $f$ and $\cF$ (i.e. $F_n|_{\omega}=\av{f}{\omega}$ for any atom $\omega$ of the algebra $\cF_n$), then $F_n$ converges to $f$ almost everywhere and in $L^1(X,\mathfrak A,\mu)$. We use the symbol $\tcF$ for the set of all atoms of~$\cF_n$,~$n\geq 0$.
For a fixed filtration
$\cF$ we introduce the class $\Class_\Omega^\cF$ of functions $\vf\colon X\to\partial\Omega_0$
such that the condition~$\av\vf\omega\in\Omega$ holds for every  $\omega\in\tcF$. 

\begin{Def}\label{AlphaExt}%\noindent{\bf Definition.}
The lens $\tilde\Omega$ is called an extension of $\Omega$ if $\Omega \subset \tilde\Omega$ and $\dfi\Omega=\dfi\tilde\Omega$.\footnote{Note that this definition slightly differs from the one given in~\cite{SZ}.}

Let $\alpha\in(0,1)$. The lens $\tilde\Omega$ is called an $\alpha$-extension of $\Omega$ if for every 
pair of points $x,y$ from $\Omega$ such that the point $z=\alpha x+(1-\alpha)y$ together
with the whole straight line segment $[z,y]$ is contained in $\Omega$\textup, we have
$[x,y]\subset\tilde\Omega$. 
\end{Def}
\begin{figure}[h!]
\includegraphics[height=6cm]{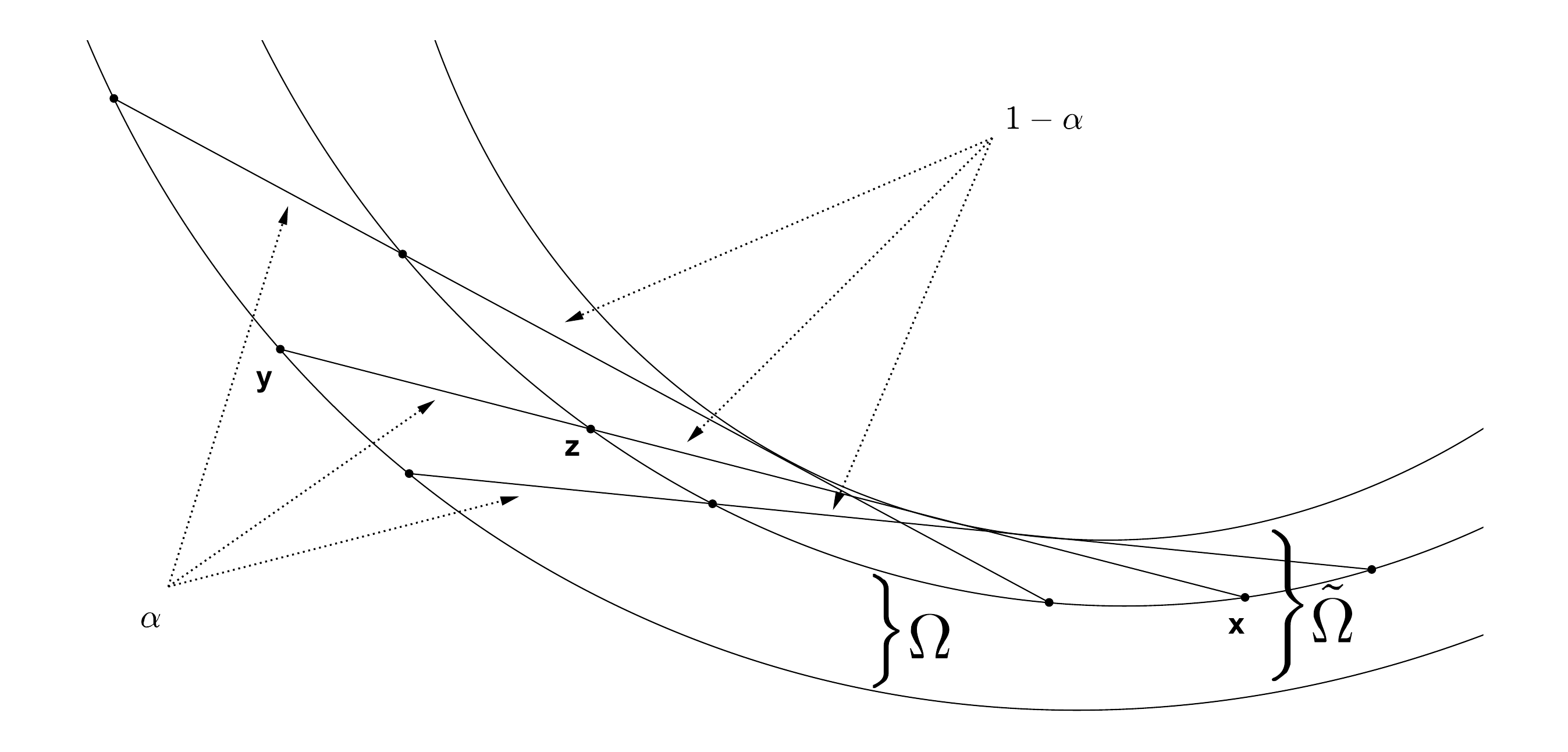}
\caption{\mbox{Illustration to Definition~\ref{AlphaExt}: 
$\tilde{\Omega}$ is an~$\alpha$-extension of~$\Omega$.}}
\label{fig:ill}
\end{figure}

Note that if $z=\beta x+(1-\beta)y$,~$\beta>\alpha$, and $[z,y]$ is contained in $\Omega$, then $[x,y]$ is contained in any 
$\alpha$-extension of $\Omega$, i.e. an $\alpha$-extension is simultaneously
a $\beta$-extension for any $\beta>\alpha$.

%\noindent{\bf Definition.}
\begin{Def}\label{AlphaFiltr}
We say that $\cF$ is an $\alpha$-filtration if for every pair of atoms~$\omega_n \in \cF_n$\textup,~$\omega_{n+1} \in \cF_{n+1}$ such that $\omega_{n+1}\subset\omega_n$ we have 
$\mu(\omega_{n+1})\ge\alpha\mu(\omega_n)$.
 We refer to the pair $(\omega_n,\omega_{n+1})$ as above as a parent and a child.
We say that $\cF$ is a binary filtration if every parent in $\cF$ has at most two children.
\end{Def}

\begin{Def}\label{AlphaMart}%\noindent{\bf Definition.}
Let $\cF$ be a binary filtration and let $\{F_n\}_{n \geq 0}$ be the martingale generated by a function 
$\vf\in\Class_\Omega^\cF$ and $\cF$. We say that $\{F_n\}_{n \geq 0}$ is an $\alpha$-martingale if for all
$\omega\in\tcF$ the following condition is fulfilled\textup: if $\omega',\omega''$ are children of 
$\omega$ and the straight line segment $[\av\vf{\omega''},\av\vf\omega]$ is not contained 
in $\Omega$\textup, then we have $|\av\vf{\omega'}-\av\vf\omega|\ge\alpha|\av\vf{\omega'}-\av\vf{\omega''}|$. 
For an arbitrary filtration $\cF$ 
we say that $\{F_n\}_{n \geq 0}$ is an $\alpha$-martingale if there exists a binary filtration $\tilde\cF = \{\tilde\cF_m\}_{m \geq 0}$\textup, 
such that the martingale $\{\tilde F_m\}_{m \geq 0}$ generated by $\vf$ and $\tilde\cF$ is an $\alpha$-martingale and $F_n=\tilde F_{m_n}$ for some increasing sequence $\{m_n\}_n$.
\end{Def}

In other words, for an $\alpha$-filtration, the condition $\mu(\omega_{n+1})\ge\alpha\mu(\omega_n)$
should be fulfilled for all children of all subsets, whereas for $\alpha$-martingales this
condition is needed only in some cases. This relation is clarified in the simple lemma below.

\begin{Le}
Let $\alpha \in (0,1)$ and  let $\cF$ be an $\alpha$-filtration. Then the martingale generated by any function $\vf\in\Class_\Omega^\cF$ and $\cF$ is an $\alpha$-martingale.
\end{Le}

\begin{proof}
If the given filtration is binary, then the statement is evident. Indeed, for any $w \in \tcF$ we get 
$$
|\av\vf\omega-\av\vf{\omega'}|=\frac{\mu(\omega'')}{\mu(\omega)}|\av\vf{\omega'}-\av\vf{\omega''}|
$$
and
$$
|\av\vf\omega-\av\vf{\omega''}|=\frac{\mu(\omega')}{\mu(\omega)}|\av\vf{\omega'}-\av\vf{\omega''}|
$$
 from the identity 
$\mu(\omega)\av\vf\omega=\mu(\omega')\av\vf{\omega'}+\mu(\omega'')\av\vf{\omega''}$
for the children $\omega',\omega''$ of $\omega$.
By Definition~\ref{AlphaFiltr}, both coefficients are not less than $\alpha$.

So, for an arbitrary $\alpha$-filtration $\cF=\{\cF_n\}_{n \geq 0}$ we need to construct a
binary one, $\tilde\cF$, such that $\cF$ is a subfiltration of $\tilde\cF$. 
We make it by induction. We start with the trivial algebra $\tilde\cF_0=\cF_0$. Suppose that the sequence $\{\tilde\cF_j\}_{j\leq m}$ has already been defined 
and satisfies two properties: for some $n$ we have $\tilde\cF_m \subset \cF_{n+1}$ and $\{\cF_i\}_{i\leq n}$ is a subfiltration of $\{\tilde\cF_j\}_{j\leq m}$;  $\av{\vf}{\omega}\in \Omega$ for any atom $\omega$ of any algebra $\tilde\cF_i$, $i \leq m$. We need to define the next algebra $\tilde\cF_{m+1}$. 
We fix some atom $\omega$ of $\tilde\cF_m$ which is not the atom of $\cF_{n+1}$ and take an
arbitrary atom $\omega'$ of $\cF_{n+1}$, such that $\omega'\subset\omega$ and $\av\vf{\omega''} \in \Omega$ %as a first child
%of $\omega$ such that 
where~$\omega''=\omega\setminus\omega'$. % (which we take as the second child) 
 
%Maximality in the above 
%sentence means that the parent of $\omega'$ is not contained in $\omega$.

The following simple consideration (see Lemma~2.3 from~\cite{SV}) guarantees the 
existence of the desired atom $\omega'$. Let us enumerate all atoms of $\cF_{n+1}$ which are subsets of $\omega$: $\omega'_1,\omega'_2,\ldots,\omega'_k$;
and let $\omega''_j=\omega\setminus\omega'_j$ for $j=1,\dots,k$. Since $\omega=\bigcup_{i=1}^k\omega'_i$
and $\omega'_i\cap\omega'_j=\varnothing$, we have
$$
\mu(\omega)\av\vf\omega=\sum_{i=1}^k\mu(\omega'_i)\av\vf{\omega'_i}=
\frac1{k-1}\sum_{i=1}^k\mu(\omega''_i)\av\vf{\omega''_i}.
$$
All the points $\av{\vf}{\omega''_i}$ belong to $\cl(\Omega_0)$ by convexity of this set. If we assume that all the points $\av\vf{\omega''_i}$ do not belong to $\Omega$,
then all of them are in~$\Omega_1$ (since~$\Omega_1$ is convex). In this case their convex
combination $\av\vf\omega$ should be in $\Omega_1$, however, it is not the case. Thus, $\av{\vf}{\omega''_i} \in \Omega$ for some $i$ and we can take $\omega'=\omega'_i$, $\omega'' = \omega''_i$. 

We now define $\tilde\cF_{m+1}$ by replacing the atom $\omega$ of~$\tilde\cF_m$ by two new atoms $\omega', \omega''$. We have made an induction step. Since the algebra $\cF_{n+1}$ is finite, after a finite number of steps we obtain $\tilde\cF_{m_{n+1}}=\cF_{n+1}$ for some $m_{n+1}$, and after that we continue the induction with $n$ increased by one.

%So, we can take the desired binary splitting $\omega=\omega'\cup\omega''$ for
%every element $\omega$ of $\tilde\cF_m$ being not in $\cF$. After finite number 
%of steps we split all elements not belonging to $\cF$ and since we took always 
%a maximal elements in our splitting we come to equality $\tilde\cF_{m_n}=\cF_n$ 
%for some subsequence $m_n$.
Clearly, the resulting binary filtration $\tilde\cF$ is an $\alpha$-filtration.
\end{proof}

\begin{Def}
We say that a positive number $\alpha$ is admissible for the filtration $\cF$ if there exist $n \geq 0$\textup, an atom $\omega \in \cF_n$\textup, and $w'\in \cF_{n+1}$ such that $\omega' \subset \omega$ and $\mu(\omega')=\alpha \mu(\omega)$.
\end{Def}

\begin{Th}\label{Th1}
Suppose that~$\tilde\Omega$ is an extension of~$\Omega$. For a fixed filtration $\cF$ and an admissible for this filtration number $\alpha$ 
the two assertions below are equivalent.
\begin{enumerate}
\item For every $\vf\in\Class_\Omega^\cF$ such that the martingale $\{F_n\}_{n\geq 0}$ generated by~$\varphi$ and~$\cF$ is an $\alpha$-martingale\textup, the monotonic rearrangement $\vf^*$ belongs to $\Class_{\tilde\Omega}$.
 
\item The domain $\tilde\Omega$ is an $\alpha$-extension of $\Omega$.
\end{enumerate} 
\end{Th}

\begin{proof}
First we prove the implication $(2)\implies(1)$. Let $\vf\in\Class_\Omega^\cF$ and let $\{F_n\}_{n\geq 0}$ be the $\alpha$-martingale
generated by $\vf$ and~$\cF$. Let $\{\tilde F_m\}$ be a binary $\alpha$-martingale such that $\tilde F_{m_n}=F_n$.
By Definition~\ref{AlphaExt}, for any $\omega\in \A{\tilde\cF}$ and its children
$\omega',\omega''$, the whole segment $[\av\vf{\omega'},\av\vf{\omega''}]$ is in $\tilde\Omega$.\footnote{In the terminology of~\cite{SZ}, the martingale $\{\tilde F_m\}$ is an $\tilde\Omega$-martingale.} 
By Theorem~3.4 from~\cite{SZ} the  monotonic rearrangement of $\tilde F_\infty = \lim \tilde F_m$ 
(that is the function $\vf^*$) belongs to $\Class_{\tilde\Omega}$. The implication is proved. 

Assume (2) is not fulfilled. Then we can find $x,y,z \in\Omega$ such that 
$z=\alpha x+(1-\alpha)y$ with $[y,z]\subset\Omega$, but $[x,y]\not\subset\tilde\Omega$, i.e. 
$[x,z]\not\subset\tilde\Omega$. Without loss of generality we can suppose that $y\in\dfi\Omega$.
Indeed, if $y\notin\dfi\Omega$, we can shift the points $y$ and $z$ to the new positions $y'$ and $z'$
along the line containing the segment $[x,y]$ so that $y'\in\dfi\Omega$, $z'=\alpha x+(1-\alpha)y'$,
and $[y',z']\subset\Omega$, but $[x,y']\not\subset\tilde\Omega$. Take two arbitrary points $a$ and $b$ on the boundary $\dfi\Omega$
such that $x\in[a,b]\subset\Omega$. Since the part of $\Omega$ between the chord $[a,b]$ and the 
corresponding arc of the boundary $\dfi\Omega$ is a convex set, the point $y$ cannot belong
to this arc. Therefore, by the letter $a$ we can  call the endpoint of the arc that is between the
points $y$ and $b$ (see Picture~\ref{fig:mart}).
\begin{figure}[h!]
\includegraphics[height=6cm]{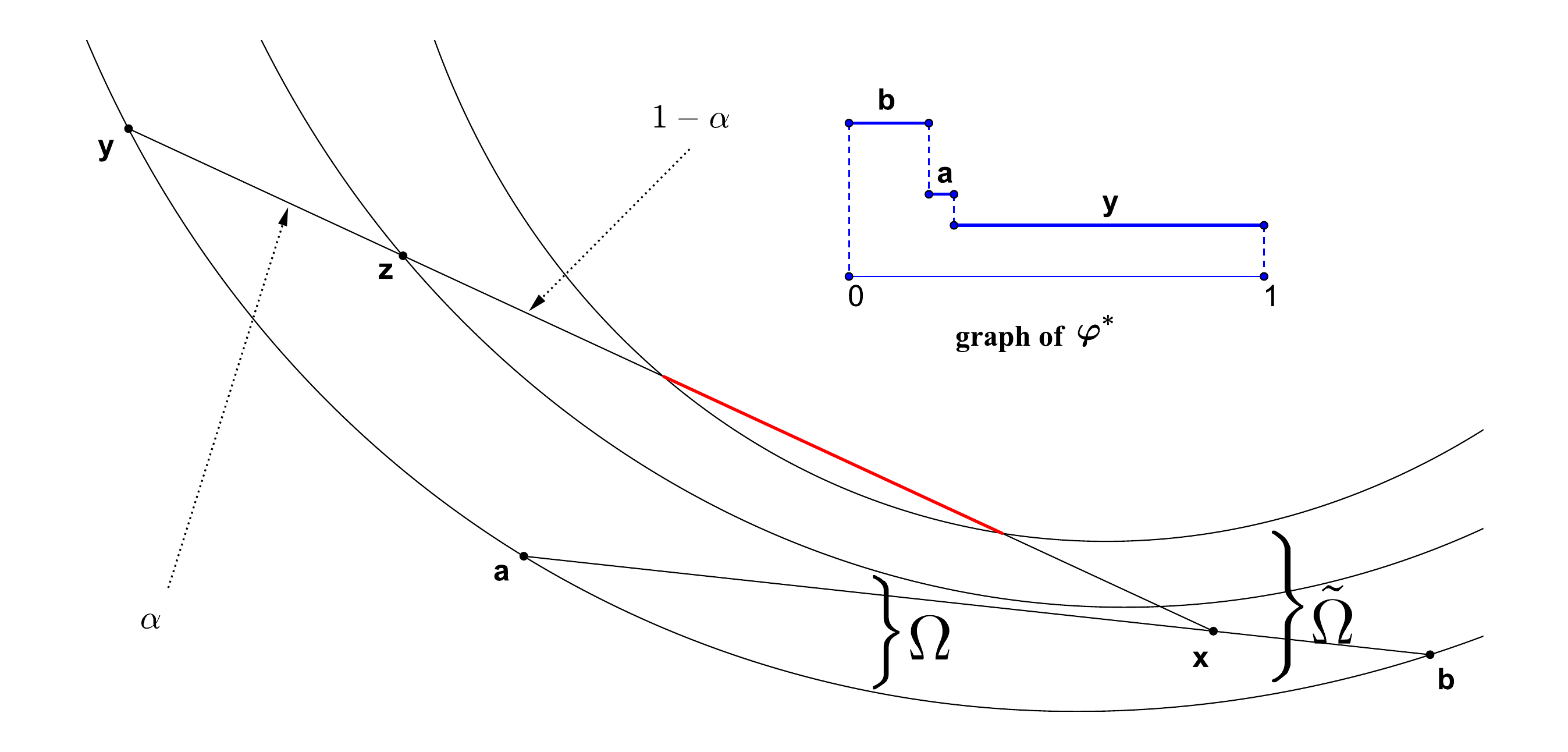}
\caption{Construction of~$\varphi$.}
\label{fig:mart}
\end{figure}

Now we take a subset $\omega\in\A{\cF}$ such that $\frac{\mu(\omega')}{\mu(\omega)}=1-\alpha$,
where $\omega'$ is a union of several children of $\omega$. Such $\omega$ and $\omega'$ do
exist because $\alpha$ is admissible for $\cF$. Define the function $\vf$ on $X$ in
the way described below. We put $\vf=y$ on $\omega'\cup(X\setminus\omega)$ and on the remaining
part $\omega''=\omega\setminus\omega'$ our function $\vf$ takes only values 
$a$ and $b$ in such a proportion that $\av\vf{\omega''}=x$. %This is possible, since by our 
%assumption $\cF$ differentiates $X$. 

Let us check that $\vf\in\Class_\Omega^\cF$. Let $\omega_1 \in \A\cF$. If $\omega_1 \cap \omega = \varnothing$, then $\av{\vf}{\omega_1}=y \in \Omega$. The average over $\omega$ is
$$
\av{\vf}{\omega}= \frac{\mu(\omega')}{\mu(\omega)}\av{\vf}{\omega'}+\frac{\mu(\omega'')}{\mu(\omega)}\av{\vf}{\omega''}= (1-\alpha) y+ \alpha x = z.
$$ 
If $ \omega_1 \supset \omega$, then 
$$
\av{\vf}{\omega_1} =\frac{\mu(\omega)}{\mu(\omega_1)}z+\frac{\mu(\omega_1)-\mu(\omega)}{\mu(\omega_1)}y \in [y,z] \subset \Omega.
$$ 
If $\omega_1 \subset \omega$, then either $\omega_1 \subset \omega'$ and $\av{\vf}{\omega_1}=y \in \Omega$ or $\omega_1 \subset \omega''$ and $\av{\vf}{\omega_1}\in[a,b]\subset \Omega$.

Let us check that $\vf^*\notin\Class_{\tilde\Omega}$. Without loss of generality we may assume that $\vf^*$ is defined on 
$[0,1]$ and it is a step function taking three values 
$b$, $a$, and $y$. For definiteness we assume that $\vf^*(0)=b$. 
Then $\av{\vf^*}{[0,\mu(\omega'')]}=x$ and the
value $\av{\vf^*}{[0,t]}$ runs through the whole segment $[x,z]$ when $t$ runs through 
$[\mu(\omega''),1]$. Since by our assumption $[x,z]\not\subset\tilde\Omega$, we
conclude that $\vf^*\notin\Class_{\tilde\Omega}$. Therefore, $(1)\implies(2)$.
\end{proof}

\section{Examples}

In this section we consider several examples of $\alpha$-extensions of lenses which correspond to famous classes of functions.

In order to be an $\alpha$-extension of~$\Omega$, $\tilde\Omega$ should contain all segments $[x,y]$ such that $x,y \in \Omega$ and $[z, y] \subset \Omega$, where $z=\alpha x + (1-\alpha)y$. It is almost obvious that this property is satisfied if and only if it is fulfilled for such segments with $y \in \dfi\Omega$ and $x,z \in \dfree\Omega$ (in what follows we will call such segments \textit{higher}). Therefore, the construction of the minimal $\alpha$-extension is quite simple with $\alpha$ and $\Omega$ at hand.

\subsection{$\BMO$  space}
It is well known that the lenses 
$$\Omega_\eps = \{(x_1,x_2)\in \mathbb{R}^2\colon x_1^2\leq x_2\leq x_1^2+\eps^2\}, \quad \eps>0,$$ correspond to the $\BMO$ space. One can easily check that a function $\vf$ lies in the $\BMO$ space on some interval and has the quadratic seminorm at most $\eps$ if and only if the function $(\vf,\vf^2)$ belongs to $\Class_{\Omega_\eps}$.

\begin{Prop}\label{propbmo}
Let $\eps, \alpha>0$. The lens $\tilde\Omega$ is an $\alpha$-extension of $\Omega_\eps$ if and only if $\tilde\Omega \supset \Omega_{\eps'}$ with $\eps' = \frac{1+\alpha}{2\sqrt{\alpha}}\eps$.
\end{Prop}
\begin{proof} 
%In order to be an $\alpha$-extension of the lens $\Omega_\eps$ the lens $\tilde\Omega$ should contain all segments $[x,y]$ such that $x,y \in \Omega_\eps$ and $[z, y] \subset \Omega_\eps$, where $z=\alpha x+(1-\alpha)y$. It is almost obvious that this property is fulfilled if and only if it is fulfilled for $y \in \dfi\Omega_\eps$ and $x,z \in \dfree\Omega$. 

Assume first that $y \in \dfi\Omega_\eps$, $x,z \in \dfree\Omega_\eps$, where $z=\alpha x+(1-\alpha)y$, $[z,y]\subset \Omega_\eps$ and the higher segment $[x,y]$ is horizontal. In such a case $y=(y_1,y_1^2)$, $z=(z_1,y_1^2)$, $x=(x_1,y_1^2)$. Then $z_1=-x_1$ and $z_1^2+\eps^2 = y_1^2$. But $y_1=\frac{z_1-\alpha x_1}{1-\alpha}=\frac{1+\alpha}{1-\alpha}z_1$, therefore, $z_1^2 \frac{(1+\alpha)^2}{(1-\alpha)^2} = z_1^2+\eps^2$, $z_1^2 = \frac{(1-\alpha)^2}{4\alpha}\eps^2$, and $y_1^2 = \frac{(1+\alpha)^2}{4\alpha}\eps^2 = \eps'^2$. Therefore, the point $(0,\eps'^2)$ lies on~$[x,y]$ and is contained in $\tilde\Omega$. Moreover,~$[x,y] \subset \Omega_{\eps'}$.

The lens $\Omega_\eps$ is invariant under the family of affine transformations $\mathrm{Aff}_t\colon (u_1,u_2) \mapsto (u_1+t,u_2+2u_1t+t^2)$, $t\in \mathbb{R}$, therefore the domain $\tilde\Omega$ should contain all the images of the point $(0,\eps'^2)$ under these maps, that is $\{(t,t^2+\eps'^2)\}_{t\in\mathbb{R}}$. Thus $\Omega_{\eps'}\subset\tilde\Omega$. Moreover,~$\Omega_{\eps'}$ is invariant under these maps as well and contains the horizontal higher segment, thus it contains all the higher segments, because they are nothing but the images of the corresponding horizontal higher segment under these affine maps. This proves that $\Omega_{\eps'}$ is an $\alpha$-extension of $\Omega_\eps$.
\end{proof}

\paragraph{\bf Proof of Corollary~\ref{corBMO}.}
Consider the filtration $\cF= \{\cF_k\}_{k\geq 0}$ on the probability space $[0,1]^n$, where $\cF_k = \{I \in \mathcal{D}\colon |I|=2^{-nk}\}$. First, note that this is $2^{-n}$-filtration. Second, note that $\|\vf\|_{\bmody}\leq \eps$ if and only if $(\vf,\vf^2)\in \Class_{\Omega_\eps}^\cF$. Theorem~\ref{Th1} states that the monotonic rearrangement operator acts from $\Class_{\Omega_\eps}^\cF$ to $\Class_{\tilde\Omega}$ if and only if $\tilde\Omega$ is an $2^{-n}$-extension of $\Omega_\eps$, which, by Proposition~\ref{propbmo}, holds exactly when $\tilde\Omega \supset \Omega_{\eps'}$ with $\eps' = \frac{1+2^n}{2^{1+n/2}} \eps$.\qed

\subsection{$A_{p_1,p_2}$ classes}

Consider the lenses 
$$\Omega_C^q = \{(x_1,x_2)\colon x_1,x_2>0, \, x_1^{q}\leq x_2 \leq C x_1^q\},\quad q\in\mathbb{R}\setminus\{0\},\;C>1.$$
These lenses are closely related to the so-called $A_{p_1,p_2}$ classes (see~\cite{Vasyunin3}) as we will see later.

We are going to find the minimal $\alpha$-extension of the lens $\Omega_C^q$. In what follows we assume $\dfi\tilde\Omega=\dfi\Omega_C^q$. We consider several cases.
\begin{Prop}\label{propap1}
Let $q>1$. Suppose that $\alpha>1-C^{\frac{-1}{q-1}}$. Then~$\tilde\Omega$ is an $\alpha$-extension of~$\Omega_C^q$ if and only if $\tilde\Omega \supset \Omega_{C'}^q$ with
\begin{equation}\label{eq3}
C' = \frac{(1-Ca^q)^q(q-1)^{q-1}}{(1-a)(a-Ca^q)^{q-1}q^q},
\end{equation}
where $a$ is the smallest of two roots of the  equation
\begin{equation}\label{eq2}
C(\alpha a+(1-\alpha))^{q}=\alpha C a^{q}+(1-\alpha).
\end{equation}
 If $\alpha\leq 1-C^{\frac{-1}{q-1}}$\textup, then the set $\{(x_1,x_2)\colon x_1>0, x_2\geq x_1^q\}$ is the minimal $\alpha$-extension of~$\Omega_C^q$.
\end{Prop}
\begin{proof} 
First, we note that a segment $[x,y]$ is the higher one if and only if 
\begin{equation}\label{eq1}
C(\alpha x_1+(1-\alpha)y_1)^{q}=\alpha C x_1^{q}+(1-\alpha)y_1^{q},
\end{equation}
 where $x=(x_1,C x_1^{q}) \in \dfree \Omega_C^q$, $y=(y_1, y_1^{q}) \in \dfi\Omega_C^q$ and 
$z=\alpha x+(1-\alpha)y \in \dfree \Omega_C^q$. 
Equation~\eqref{eq1} is homogeneous. Thus, $a = \frac{x_1}{y_1}$ satisfies equation~\eqref{eq2}. 

If $\alpha > 1-C^{\frac{-1}{q-1}}$, then equation~\eqref{eq2} has exactly two positive roots, the first is bigger than 1, and the second is smaller than 1. These two roots correspond to the higher segments with $x_1>y_1$ and $x_1<y_1$. If $\alpha \leq 1-C^{\frac{-1}{q-1}}$, then equation~\eqref{eq2} has only one root which is bigger than one. This means that there are no higher segments with $x_1<y_1$. In such a case any $\alpha$-extension of $\Omega_C^q$ should contain the union of all segments $[x,y]$ such that $x_1<y_1$ and $x \in \Omega_C^q$ and $[y,z] \subset \Omega_C^q$ for $z = \alpha x+(1-\alpha)y$, which coincides with the set $\{(x_1,x_2)\colon x_1>0, x_2\geq x_1^q\}$.
 
Let now $\alpha> 1-C^{\frac{-1}{q-1}}$. The lens $\Omega_C^q$ is invariant under the family of affine transformations $\mathrm{Aff}_t\colon (u_1,u_2) \mapsto (t u_1,t^{q}u_2)$, $t>0$, each of which preserves the property of a segment to be higher. Thus, the minimal $\alpha$-extension of $\Omega_C^q$ should be invariant under these transformations as well and therefore, should coincide with $\Omega_{C'}^q$ for some $C'>C$. For each higher segment $[x,y]$ we find the point $(t_1,t_2)$ on it, such that $t_2t_1^{-q}$ is maximal. This maximal value is exactly $C'$ defined by~\eqref{eq3}. 
\end{proof}

Arguing in the same way one can obtain the following proposition.
\begin{Prop}\label{propap2}
Let $q\leq -1$. The lens $\tilde\Omega$ is an $\alpha$-extension of $\Omega_C^q$ if and only if $\tilde\Omega \supset \Omega_{C'}^q$ with 
\begin{equation}\label{eq3b}
C' = \frac{(a-Ca^q)^{1-q}(-q)^{-q}}{(a-1)(1-Ca^q)^{-q}(1-q)^{1-q}},
\end{equation}
where $a$ is the bigger of two roots of equation~\eqref{eq2}.
\end{Prop}

Consider now $q \in (0,1)$. In order to survey $\alpha$-extensions of $\Omega_C^q$, consider the affine transformation $T\colon (u,v)\mapsto (\frac{v}{C},u)$. Then $T(\Omega_C^q) = \Omega_{C^{q'}}^{q'}$, where $q'=1/q$. The lens $\tilde\Omega$ is an $\alpha$-extension of $\Omega_C^q$ if and only if $T(\tilde\Omega)$ is an $\alpha$-extension of $\Omega_{C^{q'}}^{q'}$, but $q'>1$, therefore, one can use Proposition~\ref{propap1} to verify this property. 

For the case $q \in (-1,0)$ one can use the symmetry $T \colon (u,v) \mapsto (v,u)$ to reduce the question about an $\alpha$-extension of $\Omega_C^q$ to the question about $\Omega_{C^{-q'}}^{q'}$ with $q'=1/q$ and use Proposition~\ref{propap2} for it.

We have finished the description of $\alpha$-extensions of lenses $\Omega_C^q$ and are ready to connect it with the $A_{p_1,p_2}$ classes, where $p_1>p_2$. Recall  that a positive function $\vf$ is in $A_{p_1,p_2}$ class on an interval $J$ with the constant $Q$ if $\av{\vf^{p_1}}{I}^{1/p_1} \av{\vf^{p_2}}{I}^{-1/p_2}\leq Q $ for any subinterval $I\subset J$. Note that for $p_2>0$ this property is equivalent to the fact that the function $((Q^{-1}\vf)^{p_1},\vf^{p_2})$ is in $\Class_{\Omega_C^q}$, where $C=Q^{p_2}$, $q=\frac{p_2}{p_1}$, and for $p_2<0$  it is equivalent to the fact that the function $(\vf^{p_1},\vf^{p_2})$ is in $\Class_{\Omega_C^q}$, where $C=Q^{-p_2}$, $q=\frac{p_2}{p_1}$. Thus, the question about the monotonic rearrangement operator for dyadic-type $A_{p_1,p_2}$ classes can be investigated via $\alpha$-extensions of corresponding lenses $\Omega_C^q$. Unfortunately, for general $p_1$ and $p_2$ it does not seem possible to give a short answer for this question, but for the special case $p_1=1$ and $p_2=-1$ (which corresponds to the $A_2$ class) we can follow the above procedure and obtain a short answer similar to Corollary~\ref{corBMO} (note that it proves Corollary~\ref{corA2}).

\begin{Cor}\label{A2alphaextension}
For any $\alpha \in (0,1)$ and any $C \geq 1$ the lens $\tilde\Omega$ is an $\alpha$-extension of~$\Omega^{-1}_C$ if and only if $\tilde\Omega \supset \Omega^{-1}_{C'}$\textup, where $C' = \frac{C(\alpha+1)^2-(\alpha-1)^2}{4\alpha}$.
\end{Cor}

\end{document}